\documentclass[12pt, reqno]{amsart}
\usepackage{ternary}

\title[Ternary/Quaternary]{Ternary and Quaternary Curves of Small Fixed Genus
  and Gonality with Many Rational Points}
\author{Xander Faber}
\author{Jon Grantham}
\address{Institute for Defense Analyses \\
Center for Computing Sciences \\
17100 Science Drive \\
Bowie, MD} 
\begin{document}

\begin{abstract}
  We extend the computations from our previous paper \cite{Faber_Grantham_GF2}
  to determine the maximum number of rational points on a curve over $\FF_3$ and
  $\FF_4$ with fixed gonality and small genus. We find, for example, that there
  is no curve of genus~5 and gonality~6 over a finite field. We propose two
  conjectures based on our data. First, an optimal curve of genus~$g$ has
  gonality at most $\geomgon$. Second, an optimal curve of gonality $\gamma$ and
  large genus over $\FF_q$ has $\gamma(q+1)$ rational points.
\end{abstract}

% \tableofcontents
\maketitle

%%%%%%%%%%%%%%%%%%%%%%%%%%%%%%%%%%%%%%%%%%%%%%%%%%%%%%%%%%%%%%%%%%%%%%%%%%%%%%%%
%%%%%%%%%%%%%%%%%%%%%%%%%%%%%%%%%%%%%%%%%%%%%%%%%%%%%%%%%%%%%%%%%%%%%%%%%%%%%%%%

\section{Introduction}

Let $C$ be a smooth proper geometrically connected scheme of dimension~1 over a
finite field --- a ``curve'', for brevity. There are two well-known upper bounds
for the number of rational points on $C$. First, there is the bound of Andr\'e
Weil that arises in the study of the zeta function:
\[
\#C(\FF_q) \le q + 1 + 2g\sqrt{q},
\]
where $g$ is the genus of $C$. Define the quantity $N_q(g)$ to be the maximum
number of rational points on a curve of genus $g$ over $\FF_q$. Many techniques
have been developed to bound $N_q(g)$ and to produce examples of curves
with large numbers of rational points; see \cite{voight},
\cite{Niederreiter_Xing_book}, or \cite{Serre_Rational_Points_Book}. 
Van der Geer and van der Vlugt
\cite{vandersquared} compiled the first comprehensive table of maximal values
for small genus and field size, which eventually evolved into
\url{manypoints.org}. Table~\ref{tab:manypoints} gives the first few values of
$N_q(g)$.

  \begin{table}[ht]
\begin{tabular}{c||c|c|c}
  $g$ & $N_2(g)$ & $N_3(g)$ & $N_4(g)$ \\
  \hline
  \hline
  0 & 3 & 4 & 5 \\
  1 & 5 & 7 & 9 \\
  2 & 6 & 8 & 10 \\
  3 & 7 &10 & 14 \\
  4 & 8 & 12 & 15 \\
  5 & 9 & 13 & 17 \\
\end{tabular}
\caption{Maximum number of rational points on a curves of genus $g$ over $\FF_q$}
\label{tab:manypoints}
\end{table}

Second, if $f \colon C \to \PP^1$ is a morphism defined over $\FF_q$, then every
rational point of $C$ must map to a rational point of $\PP^1$. Consequently, we
get the bound $\#C(\FF_q) \le (\deg f)(q+1)$. This bound is optimized when we
take $f$ to have minimum degree over all such morphisms --- we call this degree
the \textbf{gonality} of $C$ and write $\gamma$ for it. This gives rise to the
``gonality-point bound'' for $C$:
\begin{equation}
  \label{eq:gonality-point}
\#C(\FF_q) \le \gamma(q+1). 
\end{equation}
We define the quantity $N_q(g,\gamma)$ to be the supremum of the number of
rational points on a curve $C$ of genus $g$ and gonality $\gamma$. (Many pairs
$(g,\gamma)$ exist for which there is no such curve, so we use the supremum.)

\begin{remark}
One might ask why we do not consider the maximum number of rational points on a
curve over $\FF_q$ of fixed gonality, while letting the genus vary freely. We
expect this quantity to be $\gamma(q+1)$ for sufficently large genus; see
Conjecture~\ref{conj:large_genus}.
\end{remark}

In our first paper \cite{Faber_Grantham_GF2}, we laid out a plan for computing the
maximum number of rational points on a curve $C$ over $\FF_q$ with small genus
and fixed gonality, and we executed this program for curves over the binary
field with genus at most 5. In the present paper, we extend those computations
to the fields $\FF_3$ and $\FF_4$. For many cases where curves of a given genus
and gonality exist, this was fairly straightforward: we either dug
through the literature to find examples of curves that met certain bounds we had
already exhibited, or else we ran general search code that we wrote for our
first paper. But in the cases where we prove non-existence results, the
increased size of the relevant search spaces requires improved algorithms and
code development. 

Table~\ref{tab:results} summarizes the quantities $N_q(g,\gamma)$ for $q \le 4$
and $g \le 5$. The entries in the column for $q = 2$ were determined in
\cite{Faber_Grantham_GF2}, while the entries for $q=3,4$ are justified in the
present paper. Recall that $N_q(g,\gamma) = - \infty$ if there is no curve of
genus $g$ and gonality~$\gamma$. The gonality of a curve over a finite field is
at most one more than the genus \cite[Prop.~2.1]{Faber_Grantham_GF2}, so we omit
entries in the table beyond $g+1$.

\begin{table}[th]
  \begin{tabular}{c|c|c|c|c}
  $g$ & $\gamma$ & $N_2(g,\gamma)$ & $N_3(g,\gamma)$ & $N_4(g,\gamma)$ \\
  \hline \hline
  0 & 1 & 3 & 4 & 5 \\
  \hline
  1 & 2 & 5 & 7 & 9 \\
  \hline
  2 & 2 & 6 & 8 & 10 \\
  \hline
  3 & 2 & 6 & 8 & 10 \\
   & 3 & 7 & 10 & 14 \\
   & 4 & 0 & 0 & 0 \\
  \hline
  4 & 2 & 6 & 8 & 10 \\
   & 3 & 8 & 12 & 15 \\ 
   & 4 & 5 & 10 & 13 \\
   & 5 & 0 & 0 & $-\infty$ \\
  \hline
  5 & 2 & 6 & 8 & 10 \\
   & 3 & 8 & 12 & 15 \\ 
   & 4 & 9 &13 & 17 \\
   & 5 & 3 & 4 & 5 \\
   & 6 & $-\infty$ & $-\infty$ & $-\infty$ \\
  \end{tabular}
  \caption{Supremum of the number of rational points on a binary, ternary, or
    quaternary curve with fixed genus and gonality.}
  \label{tab:results}
\end{table}

The final row of Table~\ref{tab:results} suggests a pattern, which we are able
to complete:

\begin{theorem}
  \label{thm:genus5_xp}
There is no curve of genus~5 and gonality~6 over a finite field.
\end{theorem}

\begin{proof}
  Suppose there were a curve $C$ of genus~5 and gonality~6 over $\FF_q$.
  Table~\ref{tab:results} shows that $q \ge 5$. Weil's lower bound yields
  \[
  \#C(\FF_{q^3}) \ge q^3 + 1 - 10q^{3/2} = (q^{3/2} - 5)^2 - 24 \ge 14, 
  \]
  so that $C(\FF_{q^3}) \ne \varnothing$. The presence of a cubic point implies
  that $C$ has gonality at most 5 by \cite[Cor.~2.5]{Faber_Grantham_GF2}.
\end{proof}

Serre introduced a powerful technique for showing that curves over finite fields
with certain numerical properties cannot exist \cite[II - The Case
  $q=2$]{Serre_Curves_Harvard}. Lauter transformed this technique into a proper
algorithm in the self-contained article \cite{Lauter_algorithm_1998}; see also
\cite{Lauter_genus5_GF3} and \cite[{\S}VII.2]{Serre_Rational_Points_Book}. The
idea is to efficiently list all real Weil polynomials (essentially zeta
functions) that could belong to a curve with given genus and number of rational
points (perhaps over extension fields). Each of these is the real Weil
polynomial of an isogeny class of abelian varieties, and one attempts to show by
arithmetic/geometric methods that there is no Jacobian variety in this class.
For example, upon applying Lauter's algorithm to the case of curves of genus~5
and gonality~6 over $\FF_4$, one finds a single real Weil polynomial:
\[
  (T - 4)^2  (T + 1)^3.
\]
None of the methods developed in
\cite{Howe_Lauter_2003,Howe_Lauter_new_methods_2012,Lauter_geometric_methods_2001}
seem able to rule out the possibility of a Jacobian in the associated isogeny
class (which would yield an alternate proof of Theorem~\ref{thm:genus5_xp}).  We
will return to this line of thought in the forthcoming paper
\cite{Faber_Grantham_XP}.

An \textbf{optimal curve} is a curve of genus $g$ over $\FF_q$ with $N_q(g)$
rational points. One expects the geometry of these to be rather special because
their arithmetic sets them apart. For example, Rigato showed that for small-genus
curves over $\FF_2$, there are very few isomorphism classes of optimal curves
\cite{Rigato_optimal_curves}. We propose a conjecture on the gonality of optimal
curves:

\begin{conjecture}[Optimal Gonality]
  \label{conj:optimal}
  Let $C$ be an optimal curve over $\FF_q$ of genus~$g$. Then $C$ has gonality
  at most $\geomgon$.
\end{conjecture}

Note that $\geomgon$ is the maximum geometric gonality of a curve $C_{/\FF_q}$
of genus~$g$ --- i.e., among morphisms $C \to \PP^1$ defined over the algebraic
closure $\bar \FF_q$, there exists one of degree at most $\geomgon$.

Our Optimal Gonality Conjecture holds for $q = 2,3,4$ and $g \le 5$ by
Tables~\ref{tab:manypoints} and~\ref{tab:results}. Rigato's work shows that it
holds for $q = 2$ and $g = 6$ as well; there are exactly two isomorphism classes
of optimal curves over $\FF_2$ of genus 6, and both of them have gonality 4. In
the appendix to \cite{Faber_Grantham_GF2}, additional examples of optimal curves
of genus 7, 8, and 9 over $\FF_2$ are given, and in each case the gonality is
strictly smaller than $\geomgon$.

In \cite{Faber_Grantham_GF2}, we posed a conjecture on the maximum number of
rational points on a curve over $\FF_2$ of gonality $\gamma$ as the genus tends
to infinity. Based on our new data, we feel emboldened to extend the statement
to all finite fields:

\begin{conjecture}
  \label{conj:large_genus}
  Fix a prime power $q$ and an integer $\gamma \ge 2$. For $g$ sufficiently
  large, $N_q(g,\gamma) = \gamma (q+1)$.
\end{conjecture}

In \S\ref{sec:hyperelliptic}, we address the elliptic and hyperelliptic entries
in Table~\ref{tab:results}. We also prove Conjecture~\ref{conj:large_genus} when
$\gamma = 2$; see Theorem~\ref{thm:hyperelliptic_asymptotic}. We quickly handle
curves of genus~3 in \S\ref{sec:genus3}.

In order to describe all curves of genus 4 or 5, we find ourselves in need of a
list of all cubic forms (genus 4) or quadratic forms (genus 5) modulo the action
of a particular orthogonal group $O(Q)$. Unfortunately, computing the orthogonal
group $O(Q)$ by a naive search as we did in \cite{Faber_Grantham_GF2} turns out
to be computationally untenable for the needs of the present paper. We sketch a
more efficient approach in \S\ref{sec:orthogonal_groups}. The idea is to view an
element $g \in O(Q)$ as a matrix of indeterminates and then observe that
equating coefficients in the relation $Q(g(x)) = Q(x)$ gives rise to a very
structured system of quadratic equations in these indeterminates.

We treat curves of genus~4 in \S\ref{sec:genus4}. The majority of our work goes
toward proving that there is no curve of genus~4 and gonality~5 over
$\FF_4$. That computation breaks into two parts: the first is a large search
implemented in C, while the second uses Sage to identify smooth curves among the
output of the first step. Finally, we address curves of genus~5 in
\S\ref{sec:genus5}. Here the bulk of our effort goes toward determining the
maximum number of points on a curve with gonality~5, and toward the
non-existence of curves with gonality~6. Viewed from a distance, these
computations were quite similar to those performed in \cite{Faber_Grantham_GF2}
in order to address curves over $\FF_2$. However, since the search spaces are so
much larger over $\FF_3$ and $\FF_4$, it was necessary to improve our algorithms
and our code. Despite these efforts, our computations still required weeks of
compute time on a multi-core machine.  We outline these algorithmic changes and
present our findings in \S\ref{sec:gonality56}.

We close with a discussion about the software used and created for this
project. For simple verification of the genus or number of rational points of
the smooth model of an algebraic curve, we used Magma \cite{magma}. For ease of
development and the ability to optimize via analysis of the source code, we used
Sage \cite{sage_8.7}. Our code is Python3 compatible, and hence will run under
Sage~9.1 as well. In order to launch many Sage jobs asynchronously, we wrote a
flexible Python3 script called \texttt{sage{\us}launcher.py} that may be of use
to other researchers. Some of our code for genus-4 computations was written in C
and depends on the FLINT library \cite{flint-2.6.0}. All of our source code can
be found at \url{https://github.com/RationalPoint/gonality}.

Throughout this article, the field $\FF_4$ will be represented as
$\FF_2[t]/(t^2+t+1)$.

%%%%%%%%%%%%%%%%%%%%%%%%%%%%%%%%%%%%%%%%%%%%%%%%%%%%%%%%%%%%%%%%%%%%%%%%%%%%%%%%
%%%%%%%%%%%%%%%%%%%%%%%%%%%%%%%%%%%%%%%%%%%%%%%%%%%%%%%%%%%%%%%%%%%%%%%%%%%%%%%%

\section{Elliptic and Hyperelliptic Curves}
\label{sec:hyperelliptic}

Let us begin with curves of genus~1. A curve $C_{/\FF_q}$ of genus~1 necessarily
has a rational point by the Weil bound: $\#C(\FF_q) \ge q + 1 - 2\sqrt{q} =
(\sqrt{q} - 1)^2 > 0$. In particular, $C$ is an elliptic curve; it is given by a
Weierstrass equation $y^2 + a_1xy + a_3y = x^3 + a_2x^2 + a_4x + a_6$ with $a_i
\in \FF_q$; and $C$ has gonality~2. We conclude that $N_q(1,2) = N_q(1)$ for all
$q$. Now we are in a position to quote a result of Serre:

\begin{theorem}[{\cite[Thm.~6.3]{Serre_Rational_Points_Book}}]
  Set $m = \lfloor 2\sqrt{q} \rfloor$. Then $N_q(1) = q + 1 + m$ except when
  $q = p^a$ for $a$ odd, $e \geq 5$, $m \equiv 0 \pmod p$, in which case
  $N_q(1) = q + m$.
\end{theorem}

\begin{corollary}
  \label{cor:Nx12}
  $N_3(1,2) = 7$ and $N_4(1,2) = 9$.
\end{corollary}

Our next task is to look at hyperelliptic curves of genus $> 1$; see
\cite[Prop.~7.4.24]{Qing_Liu_Algebraic_Geometry} for the geometry of
hyperelliptic equations. We begin by constructing two families of optimal
hyperelliptic curves parameterized by $g$ and $q$: one for odd $q$ and one for
even $q$.

If $q$ is odd and $g > \frac{q^2-2}{2}$, then we can consider the curve
\[
C_{/\FF_q} : y^2 = x^{2g+2 - q^2}(x^q - x)^q + 1.
\]
In order to see that $y^2 = P(x)$ is smooth, one checks that $\gcd(P,P') =
1$. In our case, $P'$ vanishes only at $\FF_q$-rational points, while $P$ does
not vanish there. Therefore, $C$ is a hyperelliptic curve of genus~$g$ and has
$2(q+1)$ $\FF_q$-rational points (including the two at infinity). This is
optimal: the gonality-point bound \eqref{eq:gonality-point} shows that
$N_q(g,2) \le 2(q+1)$.

If instead $q$ is even and $g \ge q - 1$, we take
\[
C_{/\FF_q} : y^2 + \big((x^q + x)x^{g + 1 - q} + 1\big)y = (x^q + x)^2.
\]
The criterion for a hyperelliptic equation $y^2 + Q(x)y = P(x)$ in
characteristic 2 to define a smooth curve is expressed as $\gcd(Q, (Q')^2P +
(P')^2) = 1$. One verifies that this holds for our curve by considering the
cases $g$ odd and $g$ even separately. Then $C$ has genus $g$ and $2(q+1)$
rational points (including the two at infinity).

Combining these two constructions, we immediately obtain

\begin{theorem}
  \label{thm:hyperelliptic_asymptotic}
  Fix a prime power $q$. Then $N_q(g,2) = 2(q+1)$ for $g \gg 0$. 
\end{theorem}

Let us turn to the case $q = 3$ more specifically. The first construction yields
a smooth hyperelliptic curve of genus $g$ with 8 rational points when $g \ge 4$.
The curve $C_{/\FF_3} : y^2 = x^{2g-1}(x^3 - x) + 1$ is smooth of genus $g$
whenever $g \not\equiv 1 \pmod{6}$, and it has 8 rational points. In particular,
this applies when $g = 2,3$, and we have proved

\begin{theorem}
  \label{thm:N3g2}
  $N_3(g,2) = 8$ for $g \ge 2$.
\end{theorem}

Finally, we treat the case $q = 4$. The second construction above gives a smooth
hyperelliptic curve with 10 rational points for each genus $g \ge 3$. 

For $q = 4$, the construction above shows that $N_4(g,2) = 10$ for $g \ge 3$. To deal with genus 2, consider the curve $y^2 + (x^3+t+1)y = x^5 +
x^2$. It is smooth with 10 rational points (including the two at infinity). Thus,

\begin{theorem}
  \label{thm:N4g2}
  $N_4(g,2) = 10$ for $g \ge 2$.
\end{theorem}

%%%%%%%%%%%%%%%%%%%%%%%%%%%%%%%%%%%%%%%%%%%%%%%%%%%%%%%%%%%%%%%%%%%%%%%%%%%%%%%%
%%%%%%%%%%%%%%%%%%%%%%%%%%%%%%%%%%%%%%%%%%%%%%%%%%%%%%%%%%%%%%%%%%%%%%%%%%%%%%%%

\section{Curves of Genus 3}
\label{sec:genus3}

We have already dealt with hyperelliptic curves in the preceding section, so all
that remains is gonality~3 and gonality~4 \cite[Prop.~2.1]{Faber_Grantham_GF2}. A non-hyperelliptic curve of genus~3
can be realized as a smooth plane quartic via the canonical
embedding. All of our examples will be of this sort. 

\begin{theorem}
  \label{thm:N333}
  $N_3(3,3) = 10$.
\end{theorem}

\begin{proof}
  Serre \cite{Serre_points_on_curves_1983} showed that $N_3(3) \le 10$, so we
  immediately have $N_3(3,3) \le 10$. Serre also exhibited the smooth plane
  quartic
  \[
    C_{/\FF_3} : y^3z - yz^3 = x^4 - x^2z^2, 
    \]
    which has 10 rational points. A non-hyperelliptic curve of genus~3 with a rational point has gonality~3 \cite[Cor.~2.3]{Faber_Grantham_GF2}, so
    $N_3(3,3) \ge 10$.
\end{proof}

\begin{theorem}
  \label{thm:N433}
  $N_4(3,3) = 14$.
\end{theorem}

\begin{proof}
  Ihara's explicit bound shows that $N_4(3) \le 14$ \cite[\S2]{Ihara_rational_points}, and Serre
  exhibited the smooth plane quartic
  \[
  C_{/\FF_4} : (x + y + z)^4 + (xy + xz + yz)^2 + xyz(x + y + z) = 0,
  \]
  which has 14 rational points. A non-hyperelliptic curve of genus~3 with a
  rational point has gonality~3 \cite[Cor.~2.3]{Faber_Grantham_GF2}, so we
  conclude that $N_4(3,3) = 14$.
\end{proof}

\begin{theorem}
  \label{thm:Nx34}
  $N_3(3,4) = 0$ and $N_4(3,4) = 0$.
\end{theorem}

\begin{proof}
  A curve of genus $g$ with a rational point has gonality at most $g$
  \cite[Prop.~2.1]{Faber_Grantham_GF2}; hence, $N_g(3,4) \le 0$.  Howe, Lauter,
  and Top \cite{HLT} produced the following examples of pointless smooth plane
  quartics:
  \begin{align*}
    C_{/\FF_3} \colon & x^4 +xyz^2 +y^4 +y^3z-yz^3 +z^4 = 0\\
    C_{/\FF_4} \colon & (x^2 +x z)^2 +t (x^2 +x z) (y^2 +y z)+(y^2 +y z)^2 +t^2 z^4 = 0.
  \end{align*}
  A non-hyperelliptic curve of genus~3 with no rational point has gonality~4
  \cite[Cor.~2.3]{Faber_Grantham_GF2}, and the theorem follows.
\end{proof}

%%%%%%%%%%%%%%%%%%%%%%%%%%%%%%%%%%%%%%%%%%%%%%%%%%%%%%%%%%%%%%%%%%%%%%%%%%%%%%%%
%%%%%%%%%%%%%%%%%%%%%%%%%%%%%%%%%%%%%%%%%%%%%%%%%%%%%%%%%%%%%%%%%%%%%%%%%%%%%%%%

\section{Interlude on the computation of orthogonal groups}
\label{sec:orthogonal_groups}

Fix a finite field $\FF_q$, and let
\[
Q(x) = \sum_{1 \leq i\leq j \leq n} c_{i,j} x_i x_j
\]
be a quadratic form in $n$ variables, where $x = (x_1, \ldots, x_n)$ and
$c_{i,j} \in \FF_q$ are not all zero. The \textbf{orthogonal group} of $Q$ is
defined to be
\[
  O(Q) = \left\{ g \in \GL_n(\FF_q) \ : \ Q(g(x)) = Q(x) \right\}.
\]
We wish to write down all of the matrices in $O(Q)$ in an efficient manner.

Let $g = (g_{i,j})$ be an $n \times n$ matrix with undetermined coefficients,
and let us impose the equality $Q(g(x)) = Q(x)$. Equating coefficients of the
quadratic monomials in the $x_i$'s on both sides gives rise to a system of
$\frac{n^2+n}{2}$ quadratic equations in the $g_{i,j}$'s:
\begin{enumerate}
\item[(a)] Equating the coefficients of the diagonal term $x_i^2$ gives the relation
  \[
    Q(g_{1i},\ldots, g_{ni}) = c_{i,i}.
    \]
\item[(b)] Write $\langle x,y\rangle = Q(x+y) - Q(x) - Q(y)$ for the associated
  bilinear form. The relation arising from the coefficients of the term $x_ix_j$
  is given by
  \[
    \langle g_{\bullet i}, g_{\bullet j} \rangle = c_{i,j}.
  \]
  In particular, it depends only on $g_{rs}$ with $s \in \{i,j\}$.
\end{enumerate}
Any invertible matrix $g = (g_{i,j})$ that satisfies properties (a) and (b) will
be an element of the orthogonal group $O(Q)$, and every element of $O(Q)$ is
obtained in this way. It follows that the output of
Algorithm~\ref{alg:ortho_group} is correct.

\begin{example}
  \label{ex:ortho4}
Let $Q(x_1,x_2,x_3,x_4) = x_1x_2 + x_3x_4 \in \FF_q[x_1,x_2,x_3,x_4]$.  If $g =
(g_{i,j})$ is a $4 \times 4$ matrix of indeterminates, then the equation
$Q(g(x_1,x_2,x_3,x_4)) = Q(x_1,x_2,x_3,x_4)$ gives rise to the following system
of 10 quadratic equations:
\begin{align*}
x_1^2:& \ g_{11}g_{21} + g_{31}g_{41} = 0 \\
x_1x_2:& \  g_{12}g_{21} + g_{11}g_{22} + g_{32}g_{41} + g_{31}g_{42} = 1 \\
x_1x_3:& \  g_{13}g_{21} + g_{11}g_{23} + g_{33}g_{41} + g_{31}g_{43} = 0 \\
x_1x_4:& \  g_{14}g_{21} + g_{11}g_{24} + g_{34}g_{41} + g_{31}g_{44} = 0 \\
x_2^2:& \  g_{12}g_{22} + g_{32}g_{42} = 0 \\
x_2x_3:& \  g_{13}g_{22} + g_{12}g_{23} + g_{33}g_{42} + g_{32}g_{43} = 0 \\
x_2x_4:& \  g_{14}g_{22} + g_{12}g_{24} + g_{34}g_{42} + g_{32}g_{44} = 0 \\
x_3^2:& \  g_{13}g_{23} + g_{33}g_{43} = 0 \\
x_3x_4:& \  g_{14}g_{23} + g_{13}g_{24} + g_{34}g_{43} + g_{33}g_{44} = 1\\
x_4^2:& \  g_{14}g_{24} + g_{34}g_{44}  = 0.
\end{align*}
\end{example}

\begin{algorithm}[ht]
  \caption{--- Compute the elements of the orthogonal group for a quadratic
    form \newline $\displaystyle Q = \sum_{1 \le i \le j \le n} c_{i,j} x_i x_j$
    over $\FF_q$.}
  \begin{algorithmic}[1]
    \STATE initialize an empty list $L$
    \FOR {$1 \le i \le n$}
    \STATE compute the set $S_i$ of solutions to the equation $Q(x) = c_{i,i}$
    \ENDFOR 
  \FOR {$(g_{11}, \ldots, g_{n1}) \in S_1$}
  \FOR {$(g_{12}, \ldots, g_{n2}) \in S_2$}
  \STATE \textbf{if} $\langle g_{\bullet 1}, g_{\bullet 2} \rangle \ne c_{1,2}$ \textbf{then} continue
  \FOR {$(g_{13}, \ldots, g_{n3}) \in S_3$}
  \STATE \textbf{if} $\langle g_{\bullet 1}, g_{\bullet 3} \rangle \ne c_{1,3}$ \textbf{then} continue
  \STATE \textbf{if} $\langle g_{\bullet 2}, g_{\bullet 3} \rangle \ne c_{2,3}$ \textbf{then} continue \newline
  $\vdots$
  \FOR {$(g_{1n}, \ldots, g_{nn}) \in S_n$}  
  \STATE \textbf{if} $\langle g_{\bullet i}, g_{\bullet n} \rangle \ne c_{i,n}$ for some $i < n$  \textbf{then} continue
  \STATE \textbf{if} $g = (g_{i,j})$ is invertible \textbf{then} append $g$ to the list $L$
  \ENDFOR 
  \newline \vdots
  \ENDFOR
  \ENDFOR
  \ENDFOR
  \STATE \textbf{return} L
\end{algorithmic}
  \label{alg:ortho_group}
\end{algorithm}

Algorithm~\ref{alg:ortho_group} begins with the precomputation of the sets $S_i
\subset \FF_q^n$; for example, one can simply loop over the elements of
$\FF_q^n$ to find solutions. By the Weil conjectures, one expects $\#S_i \approx
q^{n-1}$, so the overall search space has size approximately
$q^{n(n-1)}$. However, the checks involving the associated bilinear form yield
a substantial cutdown of the set of matrices  to test for invertibility.
    
\begin{example}
  Consider the quadratic form $Q = x_1x_2 + x_3^2 \in
  \FF_2[x_1,x_2,x_3,x_4,x_5]$. Computing $O(Q)$ via a naive loop over all $5
  \times 5$ matrices with coefficients in $\FF_2$ takes approximately 54 minutes
  in Sage on a 2.6GHz Intel Core i5 with 16GB RAM. If we instead apply
  Algorithm~\ref{alg:ortho_group}, we first precompute the sets $S_i$. Note that
  $S_1 = S_2 = S_4 = S_5$, and that $\#S_i = 16$ for all $i$. It follows that
  our search space has size $(2^4)^5$ instead of $2^{25}$, and we only test
  around $2^{14.6}$ matrices for invertibility. Our Sage implementation produced
  $O(Q)$ in less than a second.
\end{example}

%% \begin{remark}
%%   There is a certain amount of apples-to-oranges comparison happening in the
%%   preceding example. Matrix construction and polynomial evaluation in Sage
%%   dominate the runtime using the algorithm in \cite{Faber_Grantham_GF2}. After
%%   correcting for this somewhat, we still find that
%%   Algorithm~\ref{alg:ortho_group} is substantially more efficient.
%% \end{remark}

Here is another practical improvement for computing some orthogonal
groups. Suppose that $Q$ is a quadratic form in $\FF_q[x_1, \ldots, x_n]$, but
there is $m < n$ such that $x_{m+1}, \ldots, x_n$ do not appear in any monomial
with nonzero coefficient. Write $\FF_q^n = U \oplus V$, where $U$ is spanned by
the first $m$ standard basis vectors and $V$ is spanned by the remaining
$n-m$. Then Lemma~B.11 of the arXiv edition of \cite{Faber_Grantham_GF2} implies
$g \in O(Q)$ if and only if it has the form
  \[
    g = \begin{pmatrix} A & 0 \\ B & C \end{pmatrix},
   \]
where $A \in O(Q|_U)$, $B : U \to V$ is an arbitrary linear map, and $C \in
\GL(V)$. In particular, this means that we can perform the search in
Algorithm~\ref{alg:ortho_group} on a quadratic form in $m$ variables and then
build up the orthogonal group of $Q$.
%% An implementation of this approach to computing $O(Q)$ as in
%% Example~\ref{ex:ortho4} took less than $1/4$ of a second to complete.

\begin{remark}
  How does one verify that Algorithm~\ref{alg:ortho_group} has been implemented
  correctly? We provide strong evidence via two methods. First, we have several
  implementations, including a naive search over all matrices and the method
  from Appendix~B.2 of the arXiv edition of \cite{Faber_Grantham_GF2}. These can
  be applied to quadratic forms in very low dimension over small finite fields
  ($q\le 4$) to vet the code. Second, we can verify that the cardinality of our
  output is correct because L.E. Dickson computed the order of $O(Q)$ well over
  a century ago. See Chapters~VII and~VIII of \cite{Dickson_Linear_Groups_1901}.
\end{remark}

%%%%%%%%%%%%%%%%%%%%%%%%%%%%%%%%%%%%%%%%%%%%%%%%%%%%%%%%%%%%%%%%%%%%%%%%%%%%%%%%
%%%%%%%%%%%%%%%%%%%%%%%%%%%%%%%%%%%%%%%%%%%%%%%%%%%%%%%%%%%%%%%%%%%%%%%%%%%%%%%%

\section{Curves of Genus~4}
\label{sec:genus4}

Hyperelliptic curves of genus~4 were treated in Theorems~\ref{thm:N3g2}
and~\ref{thm:N4g2}. It remains to handle curves of gonality 3, 4, and 5
\cite[Prop.~2.1]{Faber_Grantham_GF2}. We easily produce upper and lower bounds
for curves over $\FF_3$, so we do these first. Afterward, we treat curves over
$\FF_4$.

%%%%%%%%%%%%%%%%%%%%%%%%%%%%%%%%%%%%%%%%%%%%%%%%%%%%%%%%%%%%%%%%%%%%%%%%%%%%%%%%

\subsection{Ternary curves}

\begin{theorem}
  \label{thm:N343}
$N_3(4,3) = 12$.
\end{theorem}

\begin{proof}
  Using Oesterl\'e's method, Serre showed that $N_3(4) = 12$, which immediately
  implies that $N_3(4,3) \le 12$. Using global class field theory, Niederreiter
  and Xing \cite{Niederreiter_Xing_cyclotomic} found the following example of an
  affine plane curve whose smooth model has genus 4 and 12 rational points:
  \[
     (y^3-y) = \frac{x^3-x}{(x^2+1)^2}.
  \]
  The rational function $x$ gives a map to $\PP^1$ of degree 3, and this curve
  would violate the gonality-point bound if it were hyperelliptic. Thus,
  $N_3(4,3) \ge 12$.
\end{proof}

By Lemma~5.1 of \cite{Faber_Grantham_GF2}, a curve of genus~4 and gonality~4
or~5 can be realized in $\PP^3 = \FF_3[x,y,z,w]$ as the intersection of a cubic
surface and the quadric surface
\[
V(Q) \colon xy + z^2 + w^2 = 0,
\]
and conversely, any smooth geometrically irreducible intersection of $V(Q)$ with
a cubic surface has genus~4 and gonality at least~4.

\begin{theorem}
  \label{thm:N344}
  $N_3(4,4) = 10$.
\end{theorem}

\begin{proof}
  The surface $Q = xy + z^2 + w^2 = 0$ has 10 rational points on it, so we
  must have $N_3(4,4) \le 10$. The cubic surface with equation
  \[
    x^2y - xyz - y^2z + xz^2 + x^2w + y^2w + xw^2 - zw^2 + w^3 = 0,
    \]
  passes through all 10 of those rational points, and its intersection with
  $V(Q)$ is smooth and geometrically irreducible. Since this curve has a
  rational point, it must have gonality exactly 4 by
  \cite[Cor.~2.4]{Faber_Grantham_GF2}. Therefore, $N_3(4,4) \ge 10$.
\end{proof}

\begin{theorem}
  \label{thm:N345}
  $N_3(4,5) = 0$.
\end{theorem}

\begin{proof}
  A curve of genus~$g$ with a rational point has gonality at most $g$ by
  \cite[Prop.~2.1]{Faber_Grantham_GF2}. It follows that $N_3(4,5) \le
  0$. Consider the curve in $\PP^3 = \Proj \FF_3[x,y,z,w]$ cut out by the
  following equations:
  \begin{align*}
    xy + z^2 + w^2 &= 0 \\
    x^3 + y^3 + y^2z + x^2w + xyw - y^2w - yzw + z^2w &= 0.
    % This used to say x^3 + y^3 + yz^2 + xw^2 - yw^2 - zw^2&= 0, which is wrong.
  \end{align*}
  It is smooth and geometrically irreducible, hence of genus 4 and gonality at
  least 4. A direct search shows that it has no $\FF_9$-rational point, so it
  must have gonality~5 \cite[Cor.~2.4]{Faber_Grantham_GF2}.
\end{proof}

\begin{remark}
  Our practice has been to look in the literature for examples of the curves we
  need before turning to computer searches. Castryck and Tuitman seem to have
  given the first example of a curve with gonality~5 over $\FF_3$
  \cite[p.15]{Castryck-Tuitman-arXiv-v2}. (The cited arXiv paper contains
  examples that were excised before publication as \cite{Castryck-Tuitman}.)  We
  opted to use our own example in the above proof because it has somewhat
  shorter defining polynomials.
\end{remark}

%%%%%%%%%%%%%%%%%%%%%%%%%%%%%%%%%%%%%%%%%%%%%%%%%%%%%%%%%%%%%%%%%%%%%%%%%%%%%%%%

\subsection{Quaternary curves}

A non-hyperelliptic curve of genus~4 over $\FF_4$ can be realized in $\PP^3 =
\FF_4[x,y,z,w]$ as the intersection of a cubic surface and one of the following
quadric surfaces:
\begin{align*}
  \text{Gonality 3: }& xy + z^2 = 0;\\
  \text{Gonality 3: }& xy + zw = 0; \text{ or}\\
  \text{Gonality 4 or 5: }& xy + z^2 + tzw + w^2 = 0.
\end{align*}
See Lemma~5.1 of \cite{Faber_Grantham_GF2}. 

\begin{theorem}
  \label{thm:N443}
  $N_4(4,3) = 15$.
\end{theorem}

\begin{proof}
  Serre showed that $N_4(4) = 15$, so it will suffice to find a curve of
  gonality~3 with 15 rational points. Consider the curve in $\PP^3$ cut out by
  the equations
  \begin{align*}
    xy + z^2 &= 0 \\
    x^3 + xyz + ty^2w + (t+1)yw^2 + w^3 &= 0.
  \end{align*}
  It is smooth, hence of genus~4 and gonality~3 by the remarks at the beginning
  of this section.  One verifies by direct search that it has 15 rational
  points.
\end{proof}

%% \begin{remark}
%%   A number of trigonal curves of genus 4 over $\FF_4$ with 15 rational points
%%   exist in the literature, but none of them visibly has a morphism to $\PP^1$ of
%%   degree~3. See, e.g., \url{manypoints.org}. 
%% \end{remark}

%% To handle curves of gonality~4, the remarks at the beginning of this section
%% show that we should focus on curves that lie on the quadric surface
%% \[
%% Q = xy + z^2 + tzw + w^2 = 0.
%% \]
%% This surface has 17 rational points, but Serre showed that a curve of genus~4
%% over $\FF_4$ has at most 15 rational points. It turns out that a curve of
%% gonality~4 cannot have even this many. 

\begin{lemma}
  \label{lem:N444_13}
  $N_4(4,4) \le 13$.
\end{lemma}

\begin{proof}
Let $C$ be a curve of genus 4 and gonality 4 over $\FF_4$, which we may realize
as the intersection of $Q = xy + z^2 + tzw + w^2 = 0$ and a cubic surface in
$\PP^3$.  Serre showed that $N_4(4) = 15$, so in order to prove the lemma, we
must rule out the possibility that a cubic surface meets $V(Q)$ smoothly and passes
through 14 or 15 rational points of $V(Q)$. Note that $V(Q)$ has 17 rational points. 

The space of cubic surfaces has (projective) dimension 19. If $F$ is a cubic
form, the ideal generated by $F$ and $Q$ is not affected by replacing $F$ with
$F + LQ$ for any linear form $L$. Consequently, by adding a suitable constant
multiple of $xQ$, we may assume that $F$ has no $x^2y$-term. Similarly, we may
kill the $xy^2$-, $z^3$-, and $w^3$-terms by adding suitable multiples of $yQ$,
$zQ$, and $wQ$, respectively. In this way, we reduce the dimension of the space of
cubic surfaces under consideration down to $15$.

Let us now rule out the possibility of a curve of gonality~4 with 15 rational
points. Insisting that a cubic surface pass through a particular rational point
is a linear condition on the coefficients of the cubic's defining
polynomial. For each choice of 15 points on $V(Q)$ --- of which there are
$\binom{17}{15} = 136$ --- we use linear algebra to find a basis for the space
of cubic surfaces that vanish at all of the points. From a naive
dimension-count, we expect the resulting space to have (projective) dimension~0,
so there is a unique such cubic surface. In fact, this turns out to be the case:
we executed this procedure in Sage and determined that none of the resulting 136
cubic surfaces meet the quadric surface $V(Q)$ in a smooth curve. The resulting
computation took under a minute on a single 2.6Ghz Intel Core i5 CPU.

A similar computation applied to the $\binom{17}{14} = 680$ choices of 14
rational points on $V(Q)$ took approximately 4.5 minutes to verify that there is no
curve of genus 4 and gonality 4 with 14 rational points.
\end{proof}

\begin{theorem}
  \label{thm:N444}
  $N_4(4,4) = 13$.
\end{theorem}

\begin{proof}
  The upper bound we want is given by Lemma~\ref{lem:N444_13}. For the lower
  bound, consider the smooth curve in $\PP^3 = \Proj \FF_4[x,y,z,w]$ that is cut
  out by the equations
  \begin{align*}
    xy + z^2 + tzw + w^2 &= 0 \\
    y^2z + xz^2 + x^2w + y^2w + yzw + z^2w + xw^2 + yw^2 &= 0.
  \end{align*}
  Direct search shows that it has 13 rational points. Thus, $N_4(4,4) \ge 13$.
\end{proof}

We spend the remainder of this section describing a computational proof of the
following non-existence result:

\begin{theorem}
  \label{thm:N445}
  $N_4(4,5) = -\infty$.
\end{theorem}

Suppose that $C$ is a curve of genus 4 over $\FF_4$ with gonality~5. Then $C$
can be realized in $\PP^3 = \Proj \FF_4[x,y,z,w]$ as the intersection of the
quadric surface
\[
Q = xy + z^2 + tzw + w^2 = 0
\]
and a cubic surface by \cite[Lem.5.1]{Faber_Grantham_GF2}. Corollary~2.4 of
\cite{Faber_Grantham_GF2} asserts that such a curve satisfies $C(\FF_{4^2}) =
\varnothing$.  Our proof has two main steps:
\begin{enumerate}
  \item[Step 1.] Loop over cubic surfaces and record those that do not
    pass through any quadratic point on $V(Q)$.
  \item[Step 2.] Intersect each of the survivors from Step 1 with the quadric
    surface $V(Q)$ and determine which, if any, are smooth and geometrically
    irreducible.
\end{enumerate}
Naively, there are $(4^{20} - 1)/3 \approx 10^{11.5}$ cubic surfaces to examine,
so we wrote a C-program to execute the search. The number of survivors was
sufficiently small ($65,280$) that we could do the various bits of commutative
algebra for Step 2 in Sage.  Ultimately, we uncovered no smooth
canonical curve $C$ of genus~4 with $C(\FF_{4^2}) = \varnothing$, which proves
the theorem.

Algorithm~\ref{alg:N445} gives a more detailed description of Step 1, which we
now discuss and justify.  Write $\vec{X} = (X_0,\ldots, X_{19})$ for the tuple
of cubic monomials in $\FF_4[x,y,z,w]$, relative to some fixed ordering.  For
any coefficient vector $\vec{c} \in \FF_4^{20} \smallsetminus \{0\}$, the dot
product $\vec{c} \cdot \vec{X}$ is a cubic form, and conversely any cubic form
can be represented by such a dot product.

\begin{definition}
  \label{def:A}
  Define the set $A \subset \FF_4^{20}$ of coefficient vectors $\vec{c}$ such
  that in the cubic form $\vec{c} \cdot \vec{X}$, the following are true:
\begin{itemize}
\item The entry corresponding to $x^3$ is 1;
\item The entry corresponding to $y^3$ is nonzero;
\item The entry corresponding to $x^2y$ is 0;
\item The entry corresponding to $xy^2$ is 0;
\item The entry corresponding to $z^3$ is 0; and
  \item The entry corresponding to $w^3$ is 0.
\end{itemize}
\end{definition}

We claim that if $C$ is a curve of gonality 5, then it may be written as the
intersection of the quadric surface $V(Q)$ and a cubic surface $V(F)$, where $F
= \vec{c} \cdot \vec{X}$ for some $\vec{c} \in A$. Indeed, the point $P = (1 : 0
: 0 : 0)$ lies on the quadric surface $V(Q)$. As $C(\FF_{4^2}) = \varnothing$,
this point does not lie on $C$, and hence $F = \vec{c}\cdot \vec{X}$ does not
vanish at $P$. That is, the $x^3$-coefficient is nonzero. Since we are only
interested in the vanishing locus of $F$, we may rescale $\vec{c}$ if necessary
so that the $x^3$-coefficient of $F$ is 1. A similar argument applied to the
point $(0,1,0,0)$ shows that we may take the $y^3$-coefficient to be nonzero.
If the $x^2y$-coefficient is $a$, then we replace $F$ with $F - axQ$ to get
another cubic with the same $x^3$- and $y^3$-coefficients, but with
$x^2y$-coefficient equal to 0. The same trick applied to an appropriate multiple
of $yQ$, $zQ$, and $wQ$ kills the coefficients on $xy^2, z^3$, and $w^3$. Note
that modifying $F$ by a multiple of $Q$ does not change its vanishing
locus. This completes the justification of the claim.

\begin{algorithm}[ht]
\caption{--- Compute a list of cubic forms $F \in \FF_4[x,y,z,w]$ such that the
  variety $V(F) \cap V(Q)$ has no point over $\FF_{4^2}$, and such that any curve
  of genus 4 and gonality $5$ is isomorphic to one of these intersections}
\begin{algorithmic}[1]
  \STATE compute representatives for the $\Gal(\FF_{4^2}/\FF_4)$-orbits in the
      set $V(Q)(\FF_{4^2}) \subset \PP^3$
  \FOR { $P \in V(Q)(\FF_{4^2})$}
  \STATE store the vector $\vec{v}_P = (X_0(P), \ldots, X_{19}(P))$ of
         cubic monomials evaluated at $P$
  \ENDFOR
  \STATE initialize an empty list $L$
  \FOR {each coefficient vector $\vec{c} \in A$ as in Definition~\ref{def:A}}
  \STATE \textbf{if} $\vec{c} \cdot \vec{v}_P \ne 0$ for every $P \in V(Q)(\FF_{4^2})$
  \textbf{then} append $F = \vec{c} \cdot \vec{X}$ to $L$
  \ENDFOR
  \STATE \textbf{return} L
\end{algorithmic}
  \label{alg:N445}
\end{algorithm}

There are 17 rational points on $V(Q)$ and $272$ quadratic non-rational
points. It follows that the set of representatives computed in the first step of
Algorithm~\ref{alg:N445} has cardinality $17 + \frac{272}{2} = 153$. Storing
the evaluated monomial vectors allows us to compute dot products rather than
cubic polynomial evaluations in the main loop.  We have already discussed why it
suffices to consider only coefficient vectors in $A$ when looking for cubic
surfaces that may contain a curve of gonality~5.  The if-statement inside the
main loop checks that $V(F)$ passes through no quadratic point of $V(Q)$; it
suffices to check this for Galois-orbit representatives since $F$ is defined
over $\FF_4$. This completes the justification of Algorithm~\ref{alg:N445}.

Our C-implementation of Algorithm~\ref{alg:N445} ran in 1.25 hours on a single
2.6Ghz Intel Core i5 with 16GB RAM, and it found $65,280$ cubic forms. 

A few practical observations are in order:
\begin{itemize}
  \item Using dot products instead of polynomial evaluations provide a huge
    savings in arithmetic operations; we learned this trick from
    \cite{Savitt_points}.
  \item The if-block in the main loop can terminate as soon as \textit{some} dot
    product is zero. If we view a cubic polynomial as a random function with
    uniform random values in $\FF_{16}$, then one would expect to evaluate
    around 16 points before seeing a zero. When we ran the code, 15.39 points
    were tested per cubic.
  \item For finite field arithmetic, we initially used the FLINT
    library. However, most of the runtime in the dot product was being spent on
    memory allocation because finite fields are implemented using polynomials
    and multi-precision integer arithmetic. To get around this obstacle, we
    created addition/multiplication tables for $\FF_4$-arithmetic and then
    performed these operations via lookup into the tables.
  \item For rapid debugging, we wrote the code to be flexible enough to work
    over $\FF_2$ or $\FF_3$ as well. Our implementation over $\FF_2$ found 104
    cubics and ran in negligible time. Over $\FF_3$, it found $2,248$ cubics and
    required 19 seconds to complete.
\end{itemize}

Now we turn to Step 2 of the computation that proves Theorem~\ref{thm:N445},
which was completed with an implementation of Algorithm~\ref{alg:N445_2}.

\begin{algorithm}[ht]
\caption{--- Compute a list of cubic forms $F \in \FF_4[x,y,z,w]$ such that any
  curve of genus 4 and gonality 5 is isomorphic to $V(F) \cap V(Q)$ for exactly one
  choice of $F$ on the list}
\begin{algorithmic}[1]
  \STATE compute the orthogonal group $O(Q)$
  \STATE initialize an empty list $M$
  \STATE $S \leftarrow$ the list of cubic forms output by Algorithm~\ref{alg:N445}. 
  \WHILE { $S$ is nonempty}
  \STATE pop the first element $F$ from $S$
  \FOR {each $g \in O(Q)$}
  \STATE compute $F(g(x,y,z,w))$, scale the coefficient on $x^3$ to be 1,
  and subtract an appropriate multiple of $Q$ to kill
  the $x^2y$- $xy^2$-, $z^3$-, and $w^3$-coefficients
  \STATE remove the resulting element from $S$
  \ENDFOR
  \STATE \textbf{if} $V(F) \cap V(Q)$ is smooth \textbf{then} append $F$ to $M$
  \ENDWHILE
  \STATE \textbf{return} $M$
\end{algorithmic}
  \label{alg:N445_2}
\end{algorithm}

Let $C$ be a curve of genus~4 and gonality~5 over $\FF_4$. We know $C = V(F)
\cap V(Q)$ for some cubic form $F$ that was output by Step~1. We may assume
without loss that $F$ is the first such form that was output, so that $F$ is
immediately added to the set $M$. If $C' = V(F') \cap V(Q)$ is another such
curve isomorphic to $C$, then there is an element $g \in \GL_4(\FF_4)$ that maps
$C$ to $C'$. Since the quadric surface containing $C'$ is unique
\cite[Lem.~5.1]{Faber_Grantham_GF2}, $g$ must lie in $O(Q)$. Thus, we have an
equality of homogeneous ideals: $(F \circ g,Q) = (F',Q)$. Consequently, there is a
linear form $H$ and $a \in \FF_4$ such that $F' = a (F \circ g) + HQ$. The while-loop
in Algorithm~\ref{alg:N445_2} identifies this relationship and removes $F'$ from
$S$. It follows that $F$, and not $F'$, appears in the list $M$.

To complete the justification of Algorithm~\ref{alg:N445_2}, we must argue that
each of its outputs yields a smooth geometrically connected scheme of
dimension~1. By construction, $C := V(F) \cap V(Q)$ is smooth for $F \in M$. And
$C$ is 1-dimensional, for otherwise $F = QH$ for some linear form $H$, which
contradicts the fact that $F$ has nonzero $x^3$-coefficient. To see that $C$ is
geometrically connected, we pass to the algebraic closure $\bar \FF_4$, where
$C$ is a $(3,3)$-divisor on the surface $V(Q) \cong \PP^1 \times \PP^1$.  If $C
= C' \cup C''$, then $C$ would have a singularity at each point of intersection
of $C'$ and $C''$.

Our Sage code performs Algorithm~\ref{alg:N445_2} in under 2 minutes. It
identified 18 ideals corresponding to 1-dimensional varieties $C$ with
$C(\FF_{4^2}) = \varnothing$, but each of these varieties is
singular. Therefore, there is no curve of genus~4 and gonality~5 over $\FF_4$.

We close with a few implementation remarks:
\begin{itemize}
  \item Sage defaults to a try/except construction to determine if a Gr\"obner
    basis or primary decomposition has already been computed for a given
    ideal. The try/except construction is painfully slow if the ``except''
    clause is executed often, so we wrote special routines that avoid these
    issues.
    
  \item As with Step~1, we tested our code over $\FF_2$ and $\FF_3$ for rapid
    debugging and verification of output. Over $\FF_2$, we found a unique curve
    of genus~4 and gonality~5 up to isomorphism. This agrees with the recent
    calculations of Xarles \cite{Xarles_genus4_GF2}. Over $\FF_3$, we found a
    unique isomorphism class of genus-4 curves with gonality~5.
\end{itemize}

%%%%%%%%%%%%%%%%%%%%%%%%%%%%%%%%%%%%%%%%%%%%%%%%%%%%%%%%%%%%%%%%%%%%%%%%%%%%%%%%
%%%%%%%%%%%%%%%%%%%%%%%%%%%%%%%%%%%%%%%%%%%%%%%%%%%%%%%%%%%%%%%%%%%%%%%%%%%%%%%%

\section{Curves of Genus~5}
\label{sec:genus5}

We have already shown that $N_3(5,2) = 8$ and $N_4(5,2) = 10$ in
\S\ref{sec:hyperelliptic}. We rapidly dispose of curves of gonality~3 and~4 with
examples in the next couple of subsections, and we turn to a description of our
computation on curves of gonality at least~5 in \S\ref{sec:gonality56}.

%%%%%%%%%%%%%%%%%%%%%%%%%%%%%%%%%%%%%%%%%%%%%%%%%%%%%%%%%%%%%%%%%%%%%%%%%%%%%%%%

\subsection{Gonality 3}

For the fields at hand, the gonality-point inequality will be a sufficient upper
bound for the number of points on a trigonal curve:
\[
\#C(\FF_q) \le 3(q+1). 
\]

\begin{theorem}
  \label{thm:N353}
  $N_3(5,3) = 12$.
\end{theorem}

\begin{proof}
  The gonality-point inequality gives a sharp upper bound. The plane curve with
  defining polynomial
  \[
     x^3y^2 - xy^4 + x^4z - x^2y^2z + y^4z - x^2yz^2 + y^3z^2 - x^2z^3
  \]
  is singular only at $(0:0:1)$, where it has a cusp. By
  \cite[Lem.~6.1]{Faber_Grantham_GF2}, we find that its normalization $C$ has
  genus 5 and is trigonal. The above plane curve passes through 12 rational
  points of the plane. Blowing up $(0:0:1)$ yields a single rational point on
  $C$, so we conclude that $N_3(5,3) \ge 12$.
\end{proof}

\begin{theorem}
  \label{thm:N453}
  $N_4(5,3) =15$.
\end{theorem}

\begin{proof}
  The gonality-point bound shows that $N_4(5,3) \le 15$. The plane curve with defining
  polynomial
  \begin{align*}
    (t+1)x^3y^2 &+x^2y^3 +txy^4 +tx^4z +tx^3yz +tx^2y^2z   \\
    + xy^3z &+ty^4 z+(t+1)x^3 z^2+(t+1)x^2 yz^2 +y^3 z^2 +x^2 z^3
  \end{align*}
  passes through 15 rational points, has a cusp at $(0:0:1)$, and no other
  singularity. Just as in the proof of Theorem~\ref{thm:N353}, its normalization
  is a trigonal curve of genus 5 and $N_4(5,3) \ge 15$.
\end{proof}

%%%%%%%%%%%%%%%%%%%%%%%%%%%%%%%%%%%%%%%%%%%%%%%%%%%%%%%%%%%%%%%%%%%%%%%%%%%%%%%%

\subsection{Gonality 4}

The upper bound $N_q(5,4) \le N_q(5)$ enables us to determine $N_q(5,4)$ for $q
= 3,4$. 

\begin{theorem}
  \label{thm:N354}
  $N_3(5,4) = 13$.
\end{theorem}

\begin{proof}
  Ritzenthaler \cite{Ritzenthaler_N35} located the following curve $C \subset
  \PP^4 = \Proj \FF_3[v,w,x,y,z]$ of genus 5 with 13 rational points, presented
  as the complete intersection of three quadric hypersurfaces:
  \begin{align*}
    vw + xy &= 0 \\
    -vx + xz - y^2 + z^2 &= 0\\
    v^2 + vx + w^2 - z^2 &= 0.
  \end{align*}
  (To get this form from Ritzenthaler's paper, set $x_1 = v, x_2 = x, x_3 = z,
  x_4 = -y, x_5 = w$.) Lemma~6.6 of \cite{Faber_Grantham_GF2} shows that $C$ has
  gonality~4. Thus, $N_3(5,4) \ge 13$. Conversely, Lauter showed that $N_3(5)
  \le 13$ \cite{Lauter_genus5_GF3}.
\end{proof}

\begin{theorem}
  \label{thm:N454}
  $N_4(5,4) = 17$.
\end{theorem}

\begin{proof}
  Howe and Lauter showed that $N_4(5) = 17$ \cite{Howe_Lauter_2003}, so it
  suffices to exhibit a curve of genus 5 and gonality 4 with 17 rational
  points. Fischer gave the following example on \url{manypoints.org} in
  2014. Let $C$ be the (smooth proper) curve birational to the affine scheme in
  $\Aff^3 = \FF_4[x,y,z]$ cut out by the equations
  \begin{align*}
    y^2 + y + x^3  &= 0 \\
    z^2 + z + yx^2 + x y^2 &= 0.
  \end{align*}
  Magma verifies that it has genus 5 and 17 rational points. If we write $E$ for
  the elliptic curve with affine equation $y^2 + y = x^3$, then we see
  immediately that $C$ is a double-cover of $E$. Consequently, $C$ has gonality
  at most~4. If $C$ has gonality $\gamma \le 3$, then the gonality-point bound
  shows that $\#C(\FF_4) \le 5\gamma \le 15$. Thus $C$ has gonality exactly 4,
  and $N_4(5,4) = 17$.
\end{proof}

%%%%%%%%%%%%%%%%%%%%%%%%%%%%%%%%%%%%%%%%%%%%%%%%%%%%%%%%%%%%%%%%%%%%%%%%%%%%%%%%

\subsection{Gonality 5 and 6}
\label{sec:gonality56}

The primary difficulty in immediately applying the code that we used to treat
curves over $\FF_2$ in \cite{Faber_Grantham_GF2} is that the relevant search
space of quadric surfaces has increased dramatically in size: from $2^{15}
\approx 10^{4.5}$ to $3^{15} \approx 10^{7.2}$ or $4^{15} \approx
10^{9.0}$. Beyond some coding tricks to speed things up and to deal with memory
management, this required three major changes to our approach:
\begin{itemize}
  \item We moved away from thinking of a curve $C$ as cut out by a triple
    $(Q_1,Q_2,Q_3)$ of quadratic forms and toward thinking of $C$ as defined by
    the homogeneous ideal generated by this triple. In practice, this cut down
    the number of choices of $Q_2$ to be considered for each of the standard
    choices of $Q_1$. Additional details are given below.
      
  \item We did not attempt to find \textit{all} curves of gonality at least
    5. Instead, we targeted curves with a particular number of points and either
    stopped searching as soon as an example was located or finished the search
    to certify that no curve with that many points exists. The available bounds
    $N_3(5) = 13$ and $N_4(5) = 17$ limited the total number of searches that
    were necessary.
    
  \item We distributed much of the computation across $48$ CPUs --- Xeon(R)
    E5-2699 v3 \@ 2.30GHz with 500GB shared memory --- each running its own Sage
    process in order to complete the searches in a reasonable amount of
    time. Despite this, some of the searches over $\FF_4$ still required
    multiple days of compute time.
\end{itemize}

We now briefly sketch the strategy used to search for curves of genus 5 and
gonality at least~5; the reader should consult \cite[\S6.4]{Faber_Grantham_GF2}
for additional details.  A non-hyperelliptic, non-trigonal curve of genus 5 can
be written as the intersection of three quadric surfaces in $\PP^4 = \Proj
\FF_q[v,w,x,y,z]$, given by quadratic forms $Q_1, Q_2, Q_3$. By
\cite[Lem.~6.6]{Faber_Grantham_GF2}, we may take $Q_1$ to be of one of the
following forms:
\begin{itemize}
\item[III.] $vw+N(x,y)$, where  $N$ is a norm form for $\FF_{q^2} / \FF_q$, or
\item[IV.] $vw + xy + z^2$.
\end{itemize}
We may use the action of the orthogonal group $O(Q_1)$ to restrict $Q_2$ to a
small set of forms. More precisely, we construct a set $A(Q_1)$ of quadratic
forms $Q_2$ such that
  \begin{enumerate}
    \item[(1)] Each ideal $\langle Q_1,Q_2\rangle$ cuts out a surface in $\PP^4$;
    \item[(2)] Every quadratic form in  the ideal $\langle Q_1,Q_2\rangle$ is of the
      same type as $Q_1$ or of type IV; and
    \item[(3)] For any quadratic form $Q$ such that the ideal $\langle Q_1,Q
      \rangle$ satisfies (1) and (2), there is a unique $Q_2 \in A(Q_1)$ and an
      element $g \in O(Q_1)$ such that $\langle Q_1, Q \circ g\rangle = \langle
      Q_Q,Q_2 \rangle$.
  \end{enumerate}
Finally, we let $B(Q_1)$ be the set of quadratic forms of the same type as $Q_1$
or of type IV.

Before beginning any of the searches for genus~5 curves, we precompute $O(Q_1)$,
$B(Q_1)$, and $A(Q_1)$, in that order. Computing $O(Q_1)$ as in
Section~\ref{sec:orthogonal_groups} is straightforward and fast. The computation
of $B(Q_1)$ is also straightforward: for each quadratic form $Q$ we can
determine its type by computing the dimension of the singular locus of the
quadric surface $V(Q)$ and the number of rational points on this surface. As
there are a huge number of forms to look at, and as this operation is easily
parallelized, we distributed it across $24$ CPUs --- Xeon(R) E5-2699 v3 \@
2.30GHz with 500GB shared memory --- each running its own Sage process. Finally,
computing $A(Q_1)$ requires one to keep track of which ideals have
already appeared in some $O(Q_1)$-orbit; this was performed on a single
CPU. Knowledge of the precomputed set $B(Q_1)$ was used to determine the type of
candidate elements of $A(Q_1)$.  See Table~\ref{table:precomp} for the sizes of
these sets and the wall time required to compute them. Note that when computing
$B(vw + xy + z^2)$, we can recycle most of the computation from $B(vw +N(x,y))$,
so the latter requires substantially less wall time to compute.

\begin{table}[p]
  \begin{tabular}{c|c || c|c||c|c||c|c}
    $q$ & $ Q_1$ & $\#O(Q_1)$ & wall time & $\#B(Q_1)$ & wall time & $\#A(Q_1)$ & wall time\\
    \hline
    \hline
    $2$ & $vw + x^2 + xy + y^2$ & 1920 & 0s & 19,096 & 10s & 11 & 11s \\
    & $vw + xy + z^2$ & 720 & 0s & 13,888 & 0s & 5 & 4s \\
    \hline
    $3$ & $vw + x^2 + y^2$ & $116,640^*$ & 3s & 5,606,172 & 13m 11s & 33 & 51m \\
    &$vw + xy + z^2$ & $51,840^*$ & 19s & 4,586,868& 16s & 16 & 24m \\
    \hline
    $4$ & $vw + x^2 + txy + y^2$ & 2,088,960 & 2m 13s & 305,230,464 & 17h 1m & 83 & 43h 47m \\
    & $vw + xy + z^2$ & 979,200 & 3m & 263,983,104 & 12m & 37 & 20h 19m
  \end{tabular}
  \caption{Size and wall time for computation of the sets $O(Q_1)$, $B(Q_1)$,
    and $A(Q_1)$. We include the data for the case $q = 2$ for comparison, both
    with the cases $q = 3,4$ as well as with the sizes in
    \cite[\S6.4]{Faber_Grantham_GF2}. The asterisk on the orthogonal group
    sizes for the case $q = 3$ indicate that we computed and counted
    $PO(Q_1) = O(Q_1) / \{\pm 1\}$.}
  \label{table:precomp}
\end{table}

With these precomputations in hand, we now give a coarse description of our
algorithm for searching for curves of genus 5 and gonality 5 or 6; see
Algorithm~\ref{alg:gonality5}. Rather than searching for all such curves, we
instead attempt to find just one such curve if it exists. In order to keep each
individual process from running too long, we target curves with a particular
number of rational points. For curves of gonality~5, the total number of
searches is limited by $N_q(5,5) \le N_q(5)$. For curves of gonality~6, we look
for pointless curves with no rational point over $\FF_{q^3}$
\cite[Cor.~2.5]{Faber_Grantham_GF2}.

\begin{algorithm}[ht]
\caption{--- Find a genus~5 curve over $\FF_q$ with (a) gonality 5 and $n$ rational points, or (b) gonality 6.}
  \begin{algorithmic}[1]
\FOR { $Q_1 \in \{vw + N(x,y), vw + xy + z^2\}$ }
\FOR { $(Q_2,Q_3) \in A(Q_1) \times B(Q_1)$ }
\STATE \textbf{if} case (a) and $\#V(Q_1,Q_2,Q_3)(\FF_q) \ne n$ \textbf{then} continue
\STATE \textbf{if} case (b) and $\#V(Q_1,Q_2,Q_3)(\FF_{q^3}) > 0$ \textbf{then} continue
\STATE \textbf{if}
  every nonzero member of the linear span of $\{Q_1,Q_2,Q_3\}$ has the same
  type as $Q_1$ or type~IV, and the variety $V(Q_1,Q_2,Q_3)$ is irreducible
  and smooth of dimension~1 \textbf{then} \textbf{return} $(Q_1,Q_2,Q_3)$
\ENDFOR
\ENDFOR
\end{algorithmic}
  \label{alg:gonality5}
\end{algorithm}

Tables~\ref{table:numpoints3} and~\ref{table:numpoints4} indicate whether we
found a curve lying on $V(Q_1)$ with a particular number of points or with
gonality~6, as well as the total wall time involved for all searches.  As a
sanity check on our code, and also to see how this new strategy compares to the
one in \cite{Faber_Grantham_GF2}, we consider the case of curves over $\FF_2$;
these details are given in Table~\ref{table:numpoints2}.

\begin{table}[p]
  \begin{tabular}{c||c|c|c|c|c|c|c}
    $Q_1$ \ $\backslash$ \ $\#C(\FF_2)$ &  0 & 1 & 2 & 3 & $\ge 4$ & gonality 6 & Wall Time\\
    \hline \hline
    $vw + x^2 + y^2$ &  \yes & \yes & \yes  & \no & \no & \no & 40s \\ 
    \hline
    $vw + xy + z^2$  & \no & \no & \no  & \yes & \no & \no & 24s \\ 
    \hline
  \end{tabular}
  \caption{Data for curves of genus~5 and gonality at least~5 over $\FF_2$. We
    indicate whether a curve of gonality~5 on $V(Q_1)$ with a specified number
    of rational points exists, whether a curve of gonality~6 on $V(Q_1)$ exists,
    and the total wall time required on 48 CPUs for all of the searches. Note
    that $\#C(\FF_2) \le N_2(5) = 9$.}
  \label{table:numpoints2}
\end{table}

\begin{table}[p]
  \begin{tabular}{c||c|c|c|c|c|c|c|c}
    $Q_1$ \ $\backslash$ \ $\#C(\FF_3)$ &  0 & 1 & 2 & 3 & 4 & $ \geq 5$ & gonality 6 & Wall Time \\
    \hline \hline
    $vw + x^2 + y^2$ &  \yes & \yes & \yes  & \yes & \no & \no & \no & 59m \\ 
    \hline
    $vw + xy + z^2$  & \no & \no & \no  & \no & \yes & \no & \no & 21m \\ 
    \hline
  \end{tabular}
  \caption{Data for curves of genus~5 and gonality at least~5 over $\FF_3$. We
    indicate whether a curve of gonality~5 on $V(Q_1)$ with a specified number
    of rational points exists, whether a curve of gonality~6 on $V(Q_1)$ exists,
    and the total wall time required on 48 CPUs for all of the searches. Note
    that $\#C(\FF_3) \le N_3(5) = 13$.}
  \label{table:numpoints3}
\end{table}

\begin{table}[p]
  \begin{tabular}{c||c|c|c|c|c|c|c|c|c}
    $Q_1$ \ $\backslash$ \ $\#C(\FF_4)$ &  0 & 1 & 2 & 3 & 4 & 5 & $ \geq 6$ & gonality 6 & Wall Time \\
    \hline \hline
    $vw + x^2 + txy + y^2$ &  \yes & \yes & \yes  & \yes & \yes & \no & \no & \no & 305h 52m \\ 
    \hline
    $vw + xy + z^2$  & \no & \no & \no  & \no & \no & \yes & \no & \no & 43h 26m \\ 
    \hline
  \end{tabular}
  \caption{Data for curves of genus~5 and gonality at least~5 over $\FF_4$. We
    indicate whether a curve of gonality~5 on $V(Q_1)$ with a specified number
    of rational points exists, whether a curve of gonality~6 on $V(Q_1)$ exists,
    and the total wall time required on 48 CPUs for all of the searches. Note
    that $\#C(\FF_4) \le N_4(5) = 17$.}
  \label{table:numpoints4}
\end{table}

Looking at the tables, we immediately conclude the following:

\begin{theorem}
  \label{thm:N56}
  \[
  N_3(5,5) = 4; \quad N_3(5,6) = -\infty; \quad N_4(5,5) = 5; \quad \text{and} \quad N_4(5,6) = -\infty.
\]
\end{theorem}

\begin{example}
  Our search uncovered the following example of a curve of genus 5 and gonality
  5 over $\FF_3$ with 4 rational points:
  \begin{align*}
    Q_1 &= vw + xy + z^2 \\
    Q_2 &= x^2 + wy + vz - xz \\
    Q_3 &= vw - wx + vy + xy + vz + xz + yz
  \end{align*}
\end{example}

\begin{example}
  We located the following example of a curve of genus 5 and gonality 5 over
  $\FF_4$ with 5 rational points:
  \begin{align*}
    Q_1 &= vw + xy + z^2 \\
    Q_2 &= vz + wy + x^2 + tz^2 \\
    Q_3 &= v^2 + vw + wz + txy + xz + y^2 + z^2
  \end{align*}
\end{example}

%% Staring at Tables~\ref{table:numpoints2},~\ref{table:numpoints3},
%% and~\ref{table:numpoints4}, we are led to ask the following question:

%% \begin{question}
%%   If $C_{/\FF_q}$ is a curve of genus 5 and gonality 5 with $N_q(5,5)$ rational
%%   points, can we conclude that $C$ lies on no singular quadric surface in $\PP^4$?
%% \end{question}

%%%%%%%%%%%%%%%%%%%%%%%%%%%%%%%%%%%%%%%%%%%%%%%%%%%%%%%%%%%%%%%%%%%%%%%%%%%%%%%%
%%%%%%%%%%%%%%%%%%%%%%%%%%%%%%%%%%%%%%%%%%%%%%%%%%%%%%%%%%%%%%%%%%%%%%%%%%%%%%%%

\begin{comment}
  README
  
  Modules/Scripts used:
  progress.py
  sage_launcher.py
  search_tools.py
  
  genus4_optimal_search.sage
  
  genus4 (single-threaded c-code)
  genus4_gonality5_finish.sage

  singularity_dimension.sage
  point_count.sage
  genus5_search.sage
\end{comment}

\bibliographystyle{plain}
\bibliography{ternary}
\end{document}